\documentclass{article}

\usepackage[british]{babel}
\usepackage[applemac]{inputenc}
\usepackage{amsmath}
\usepackage{mathtools}
\usepackage{amssymb}
\usepackage{amsthm}
\usepackage{mathrsfs}
\usepackage{xcolor}
\usepackage{enumitem}
\usepackage[margin = 3.5cm]{geometry}
\usepackage{hyperref}
\usepackage[nameinlink]{cleveref}

\newtheorem{theorem}{Theorem}

\newtheorem{lemma}[theorem]{Lemma}
\newtheorem{corollary}[theorem]{Corollary}

\theoremstyle{definition}
\newtheorem{definition}[theorem]{Definition}

\theoremstyle{remark}

\DeclarePairedDelimiterX{\norm}[1]{\lVert}{\rVert}{#1}

\newcommand{\F}{\ensuremath{\mathcal{F}}}
\newcommand{\G}{\ensuremath{\mathcal{G}}}

\newcommand{\N}{\ensuremath{\mathbb{N}}}

\newcommand{\R}{\ensuremath{\mathbb{R}}}

\newcommand{\rk}{\text{rk}}
\newcommand{\rkplus}{\text{rk}_+}

\newcommand{\tr}{\text{tr}}

\newcommand{\msr}{\ensuremath{\text{msr}}}
\newcommand{\xc}{\ensuremath{\text{xc}}}
\newcommand{\rc}{\ensuremath{\text{rc}}}

\newcommand{\chivec}{\ensuremath{\chi_{\text{vec}}}}

\newcommand{\p}{\text{P}}

\newcommand{\bangle}[1]{\left\langle #1 \right\rangle}
\newcommand{\inprod}[2]{\bangle{#1, #2}}

\title{Orthonormal representations, vector chromatic number, and extension complexity}
\author{Igor Balla\thanks{Email:
    \href{mailto:iballa1990@gmail.com} {\nolinkurl{iballa1990@gmail.com}}. Research supported by SNSF Project 184522.}
}
\date{}

\begin{document}
\maketitle

\begin{abstract}
We construct a bipartite generalization of Alon and Szegedy's nearly orthogonal vectors, thereby obtaining strong bounds for several extremal problems involving the Lov\'asz theta function, vector chromatic number, minimum semidefinite rank, nonnegative rank, and extension complexity of polytopes. In particular, we derive a couple of general lower bounds for the vector chromatic number which may be of independent interest.
\end{abstract}

\section{Introduction}

Alon and Szegedy \cite{AS99} showed that there exists a constant $\delta > 0$ such that for any $t, d \in \N$ with $t \geq 3$, there exists a set of at least $d^{ \delta \log{t} / \log{\log{t}}  }$ vectors in $\R^d$ such that among any $t$ vectors, some pair is orthogonal. The aim of this note is to give a more general, bipartite version of this construction and to discuss its various implications. To this end, we will need the following definitions. Note that in the following, if we do not specify the base of the $\log$ function, then it is assumed to be in base 2. Also, for any graph $H$, we say that a graph is \emph{$H$-free} if it doesn't have $H$ as a subgraph.


\begin{definition} \label{def_orthonormal}
Let $R$ be a Euclidean space with inner product $\inprod{\cdot}{\cdot}$. An assignment of vectors $f : V(G) \rightarrow R$ to the vertices of a graph $G$ is called an \emph{orthonormal representation} of $G$ if for all distinct $u, v \in V(G)$, $f(v)$ is a unit vector and $\inprod{f(u)}{f(v)} = 0$ when $uv \notin E(G)$. Moreover, we call such a representation \emph{faithful} if it additionally satisfies $\inprod{f(u)}{f(v)} \neq 0$ when $uv \in E(G)$.
\end{definition}

\begin{definition} \label{def_msr}
For any graph $G$, we define its \emph{minimum semidefinite rank} $\msr(G)$ to be the minimum dimension of a Euclidean space $R$ such that there exists an orthonormal representation $f : V(G) \rightarrow R$. Moreover, we define $\msr_f(G)$ to be the same minimum over faithful orthonormal representations.

\end{definition}


The original motivation for Alon and Szegedy's construction was to give a counterexample to a conjecture of F\"uredi and Stanley \cite{FS92} regarding a problem of Erd\H{o}s which was almost\footnote{The only difference is that orthonormal representations allow multiple vertices to be labeled with the same vector, while Erd\H{o}s wanted all vectors to be distinct. This difference isn't substantial when considering orthonormal representations of $H$-free graphs since no vector can appear $|H|$ times and so one can remove duplicate vectors while only losing a multiplicative factor of $|H|$ in the number of vectors. } equivalent to that of determining the minimum of $\msr_f(G)$ over all $K_t$-free graphs $G$ on $n$ vertices, where $K_t$ is the complete graph on $t$ vertices. Indeed, their construction yields $K_t$-free graphs $G$ on $n$ vertices with $\msr_f(G) \leq n^{ O\left( \frac{\log{\log{t}}}{ \log{t}} \right) }$ and our main theorem generalizes this result to $K_{t,t}$-free graphs $G$, where $K_{t,t}$ denotes the complete bipartite graph with parts of size $t$.

\begin{theorem} \label{thm_main}
There exists a constant $C > 0$ such that for all integers $t \geq 3$ and $n \geq 2$, there exists a $K_{t,t}$-free graph $G$ with $n$ vertices and at least $n^{2 - C \frac{ \log{\log{t}} }{ \log{t}}}$ edges satisfying
\[
\msr_f(G) \leq n^{C \frac{ \log{\log{t}} }{ \log{t}} }.
\]
\end{theorem}

Note that the minimum of $\msr_f(G)$ over all $H$-free graphs $G$ on $n$ vertices remains the same if we replace $\msr_f(G)$ with $\msr(G)$, since one can always remove edges from the graph of an orthonormal representation in order to make the representation faithful without destroying the $H$-free property. In a previous work together with Letzter and Sudakov \cite{BLS20}, we studied this problem for various $H$, as well as a related extremal problem involving the Lov\'asz theta function. This parameter was first introduced by Lov\'asz \cite{L79} in order to determine the Shannon capacity of graphs and it has many equivalent formulations. Moreover, it is efficiently computable via semidefinite programming, so that it also has algorithmic applications and has been extensively studied, see e.g.\ Knuth \cite{K94} for more information. In order to give a definition, we let $\R^{S \times S}$ denote the space of all real matrices indexed by $S \times S$ and for any matrix $M$, we let $\lambda_1(M)$ denote its largest eigenvalue. For a graph $G$, we let $\overline{G}$ denote its complement.
 
\begin{definition} \label{def_lovasz}
The \emph{Lov\'asz theta function} $\vartheta(G)$ of a graph $G$ is defined to be the maximum over all orthonormal representations $f$ of $\overline{G}$, of the largest eigenvalue $\lambda_1(M)$ of the Gram matrix $M \in \R^{V(G) \times V(G)}$ defined by $M(u,v) = \inprod{f(u)}{f(v)}$ for $u,v \in V(G)$. 
\end{definition}

\begin{definition} \label{def_extremal}
For any graph $H$, we define $\lambda(n, H)$ and $\mu(m,H)$ to be the maximum of $\vartheta\left( \overline{G} \right)$ over all $H$-free graphs $G$ with $n$ vertices and $m$ edges, respectively.
\end{definition}


In \cite{BLS20}, we demonstrated that for some sparse graphs $H$ including cycles and certain bipartite graphs, good lower bounds for $\lambda(n, H)$ come from well-known dense and regular $H$-free graphs which are optimally pseudorandom\footnote{By optimally pseudorandom, we mean that when the graph is $d$-regular, all eigenvalue of its adjacency matrix besides $d$ are $O\left(\sqrt{d}\right)$ in absolute value.}. However, this evidently stops being the case when $H$ is a complete graph, since such constructions cannot have $\vartheta\left(\overline{G}\right)$ be larger than $O(\sqrt{n})$, while Feige \cite{F95} showed that $\lambda(n, K_t) \geq n^{1 - O\left(1/\log{t} \right)}$. An interesting case left open in our previous work was to determine what happens for a complete bipartite graph $H = K_{t,t}$ when $t \rightarrow \infty$ and in this note, we resolve this question by showing that $\lambda(n, K_{t,t}) \geq n^{1 - o(1)}$. Since our argument for proving \Cref{thm_main} generalizes the approach of Alon and Szegedy \cite{AS99}, which is in turn based on the approach of Feige \cite{F95} using the randomized graph products technique of Berman and Schnitger \cite{BS92}, it is not too surprising that \Cref{thm_main} directly implies that $\lambda(n, K_{t,t}) \geq n^{1 - o(1)}$ as $t \rightarrow \infty$. Indeed, Lov\'{a}sz \cite{L79} showed that $\vartheta(G) \leq \msr(G)$ and $n \leq \vartheta(G) \vartheta(\overline{G})$, so that we have
\begin{equation} \label{eq_theta_rank}
\vartheta\left(\overline{G}\right) \cdot \msr(G) \geq n
\end{equation}
and hence, using the fact that $\msr(G) \leq \msr_f(G)$, \Cref{thm_main} has the following immediate corollary.

\begin{corollary} \label{cor_lovasz}
There exists a constant $C > 0$ such that for all integers $t \geq 3$ and $n \geq 2$, there exists a $K_{t,t}$-free graph $G$ with $n$ vertices and at least $n^{2 - C \frac{ \log{\log{t}} }{ \log{t}}}$ edges satisfying
\[
\vartheta\left( \overline{G} \right) \geq n^{1 - C \frac{\log{\log{t}}}{\log{t}} }.
\]
\end{corollary}

One reason why it is interesting to study $\lambda(n,H)$ and $\mu(m, H)$ comes from our recent work together with Janzer and Sudakov \cite{BJS23}, in which we showed that any graph $G$ with $m$ edges has a cut of size
\begin{equation} \label{eq_random_hyper}
\frac{m}{2} + \Omega\left( \frac{m}{\vartheta\left( \overline{G} \right)} \right).
\end{equation}
Furthermore, we showed that if $\lambda(n, H) \leq O\left(n^{1/\alpha}\right)$ then $\mu(m, H) \leq O\left(m^{1/(\alpha + 1)} \right)$. As a consequence, we used upper bounds on $\lambda(n,H)$ obtained in \cite{BLS20} together with \eqref{eq_random_hyper} in order to give simple proofs of most of the known lower bounds regarding the problem of determining the maximum cut of $H$-free graphs with $m$ edges. Note that his problem has been studied extensively, see e.g. \cite{ABKS03, GJS23}, and perhaps the most important conjecture regarding it is that for any fixed $H$, there exists an $\varepsilon > 0$ such that any $H$-free graph with $m$ edges has a cut of size $m/2 + \Omega\left( m^{3/4 + \varepsilon} \right)$. In fact, it would already be interesting if a weaker version of this conjecture holds in which the exponent $3/4$ is replaced by anything bigger than $1/2$. However, we show that even this weaker conjecture cannot be proved for $K_{t,t}$-free graphs using \eqref{eq_random_hyper}. Indeed, since any graph with $n$ vertices has at most $\binom{n}{2}$ edges,  \Cref{cor_lovasz} implies that for any $\varepsilon > 0$, if $t$ is sufficiently large then $\mu\left(m, K_{t,t}\right) \geq \Omega\left( m ^{1/2 - \varepsilon }\right)$.

\begin{definition} \label{def_chi_vec}
The \emph{vector chromatic number} $\chivec(G)$ of a graph $G$ is defined to be the minimum $\kappa \geq 2$ for which there exists a Euclidean space $R$ with inner product $\inprod{\cdot}{\cdot}$ and an assignment of unit vectors $f : V(G) \rightarrow R$ to the vertices of $G$ such that $\inprod{f(u)}{f(v)}_R \leq \frac{-1}{\kappa - 1}$ for all edges $uv \in E(G)$.
\end{definition}

In \cite{BJS23}, we also proved a stronger version of \eqref{eq_random_hyper} where $\vartheta\left(\overline{G}\right)$ is replaced with the vector chromatic number $\chivec(G)$. Note that this parameter is known to be equivalent to Schrijver's theta function \cite{S79} applied to the complement graph. Indeed, a proof of this fact can be obtained by modifying the argument used by Karger, Motwani, and Sudan \cite{KMS98} to show that the Lov\'asz theta function applied to the complement of a graph is equivalent to a strict\footnote{By strict we mean that the condition in \Cref{def_chi_vec} should be changed to $\inprod{f(u)}{f(v)}_R = \frac{-1}{\kappa - 1}$.} version of the vector chromatic number. In view of the above, it follows that $\chivec(G) \leq \vartheta\left(\overline{G}\right)$ and so it is natural to ask whether $\chivec(G)$ for a $K_{t,t}$-free graph $G$ with $m$ edges could be significantly smaller than $\mu\left(m, K_{t,t}\right)$. By proving the following general lower bound for the vector chromatic number, we show that this is not the case. 

\begin{theorem} \label{thm_vec_lovasz}
For any graph $G$ with $n$ vertices, 
\[
\chivec(G) \geq \frac{ \vartheta\left( \overline{G} \right)^2 }{ n }.
\]
\end{theorem}

Indeed, applying \Cref{thm_vec_lovasz} to the graphs constructed in \Cref{cor_lovasz}, we conclude the following.

\begin{corollary} \label{cor_vec_gap}
There exists a constant $C > 0$ such that for all integers $t \geq 3$ and $n \geq 2$, there exists a $K_{t,t}$-free graph $G$ with $n$ vertices and $m$ edges such that $m \geq n^{2 - C \frac{ \log{\log{t}} }{ \log{t}}}$ and
\[
\chivec(G) \geq n^{1 - C \frac{\log{\log{t}}}{ \log{t}} } \geq m^{1/2 - C \frac{\log{\log{t}}}{ \log{t}} }.
\]
\end{corollary}

We also note that \Cref{cor_vec_gap} can be obtained directly from \Cref{thm_main} via the following inequality, which generalizes \eqref{eq_theta_rank} and may be of independent interest. 

\begin{theorem} \label{thm_ng}
For any graph $G$ on $n$ vertices, 
\[
\chivec(G) \cdot \msr(G) \geq n.
\]
\end{theorem}

\begin{definition}
Given a $d$-dimensional polytope $P$, its \emph{extension complexity} $\xc(P)$ is defined to be the minimum number of facets of a polytope $P' \subseteq \R^{d'}$ such that $P$ is the projection of $P'$ onto some $d$-dimensional subspace. 

\end{definition} 

Shortly after proving \Cref{thm_main}, we learned that Kwan, Sauermann, and Zhao \cite{KSZ22} constructed an $n^{o(1)}$-dimensional polytope with at most $n$ vertices and extension complexity at least $n^{1-o(1)}$. We show that  \Cref{thm_main} can be combined with their argument in order to give an alternate construction of such a polytope, with a slight improvement in the $o(1)$ term.

\begin{definition}
A real-valued matrix $M$ is called \emph{nonnegative} if all of its coordinates $M(i,j)$ are nonnegative and given such a matrix, its \emph{nonnegative rank} $\rkplus(M)$ is defined to be the minimum $r$ for which there exist nonnegative matrices $W_1, \ldots, W_r$ such that $M = \sum_{i=1}^{r}{W_i}$. 

\end{definition} 

Kwan, Sauermann, and Zhao obtain the appropriate polytope by first constructing an $n \times n$ matrix with a ratio of $n^{1-O\left( \frac{ \log{\log{n}} }{ \sqrt{\log{n}} } \right)}$ between its nonnegative rank and its rank, in particular answering a question of Hrube\v{s} \cite{H12}. While their argument is not probabilistic, it does seem to be related to that of \Cref{thm_main} since they use a two-family forbidden intersection theorem due to Sgall \cite{S99}, while we use a two-family forbidden intersection theorem due to Frankl and R\"{o}dl \cite{FR87}. Furthermore, our construction yields a slightly better bound in the exponent, as follows.

\begin{theorem} \label{thm_rank_plus}
There exists a constant $C > 0$ such that for all integers $n \geq 3$, there exists an $n \times n$ nonnegative matrix $M$ such that
\[
\frac{\rkplus(M)}{\rk(M)} \geq n^{1 - C \sqrt{ \frac{\log{\log{n}}}{\log{n}} } }.
\]
\end{theorem}

\begin{corollary} \label{cor_extension}
There exists a constant $C > 0$ such that for all integers $n \geq 3$, there exists a polytope $P$ with at most $n$ vertices, dimension at most $n^{ C \sqrt{ \frac{\log{\log{n}}}{\log{n}} } }$, and extension complexity 
\[
\xc(P) \geq n^{1 - C \sqrt{ \frac{\log{\log{n}}}{\log{n}} } }.
\]
\end{corollary}

\section{Proofs}

We will first present the proof of our main result, \Cref{thm_main}. Note that Alon and Szegedy's argument \cite{AS99} relies on a theorem of Frankl and R\"{o}dl \cite{FR87}, which states that there exists $\varepsilon > 0$ such that if 4 divides $k$, then any subset of $ \{+1, -1\}^{k}$ of size $2^{(1-\varepsilon)k}$ has an orthogonal pair. In order to prove \Cref{thm_main}, it would suffice to obtain a two-family version of this result. However, it is not hard to see that such a theorem cannot hold. Indeed, the middle layer subfamilies $\{v \in \{+1, -1\}^{k} : \sum_{i}{v(i)} = 0 \}$ and $\{v \in \{+1, -1\}^{k} : \sum_{i}{v(i)} = 2 \}$ are both of size $2^{(1 - o(1))k}$ and have no orthogonal pairs between them by considering the inner products modulo 4. Thus we shall restrict ourselves to only the middle layer of $ \{+1, -1\}^{k}$ and then apply the following two-family forbidden intersection theorem, which is also due to Frankl and R\"{o}dl \cite{FR87}.

\begin{theorem} \label{t:FR}
Suppose $0 < \eta < \frac{1}{4}$ and two families $\F, \G$ of subsets of $[k]$ are given which satisfy $|A \cap B| \neq \ell$ for $A \in \F, B \in \G$. If $\eta k \leq \ell \leq \left( \frac{1}{2} - \eta \right) k$, then
\[ |\F||\G| \leq 2^{2k(1 - \varepsilon )},\]
where $\varepsilon$ is a positive constant depending only on $\eta$.
\end{theorem}

In order to lower bound the number of edges in the graph we construct, we will also need the following lemma.

\begin{lemma} \label{lem_lovasz_edges}
For any graph $G$ with $n$ vertices and $m$ edges, we have
\[
\vartheta\left( \overline{G} \right)^2 \leq 2m + n. 
\]
\end{lemma}
\begin{proof}
If we let $M$ be the Gram matrix corresponding to an orthonormal representation of $G$, then
\[
\lambda_1(M)^2 \leq \tr(M^2) = \sum_{u,v \in V(G)}{M(u,v)^2} \leq 2m + n. \qedhere
\]
\end{proof}

Given vectors $v_1, \ldots, v_m \in \R^k$, recall that their tensor product $v_1 \otimes \ldots \otimes v_m$ can be seen as a vector in $\R^{k^m}$. Also, given families $A_1, \ldots, A_m \subseteq \R^k$, we define $A_1 \otimes \ldots \otimes A_m = \{ v_1 \otimes \ldots \otimes v_m : v_i \in A_i\}$. Moreover, in the following argument, we will implicitly assume all parameters are functions of $t$ and we will write $f = o(1)$ if $\lim_{t \rightarrow \infty}{f(t)} = 0$. 

\begin{proof}[Proof of \Cref{thm_main}]
We first note that we may assume $t$ is sufficiently large. Indeed, since $\msr(G) \leq n$ for any graph on $n$ vertices, we can choose $C$ large enough so that the theorem holds for all $t$ sufficiently small. Now let $k, m$ be positive integers to be chosen later with $k$ divisible by 4 and define $M_k = \left\{ v \in \{+1, -1\}^{k} : \sum_{i}{v(i)} = 0  \right \} $ to be the middle level of $\{+1,-1\}^k$, so that $|M_k| = \binom{k}{k/2} = 2^{(1 - o(1))k}$ and define $Q = \underbrace{M_k \otimes \ldots \otimes M_k}_\text{m}$ to be the tensor product of $m$ copies of $M_k$, so that we have $|Q| = |M_k|^m = 2^{km(1-o(1))}$. 

Let $S_1, \ldots, S_m, T_1, \ldots, T_m \subseteq M_k$ and let $S = S_1 \otimes \ldots \otimes S_m, T =  T_1 \otimes \ldots \otimes T_m$. We call the pair $(S,T)$ \emph{dangerous} if $u,v$ are non-orthogonal for all $u \in S, v \in T$. Note that for all $u = u_1 \otimes \ldots \otimes u_m \in S$ and $v = v_1 \otimes \ldots \otimes v_m \in T$, we have
\[ \inprod{u}{v} = \inprod{u_1}{v_1} \inprod{u_2}{v_2} \ldots \inprod{u_m}{v_m},\]
and thus $(S,T)$ is a dangerous pair if and only if $x, y$ are non-orthogonal for all $x \in S_i, y \in T_i$. For any $x \in \{+1, -1\}^k$, we let $A_x = \{ i \in [k] : x(i) = 1\}$ and observe that if $x, y \in M_k$ then 
\[ 
\inprod{x}{y} = k - 2 |A_x \triangle A_y|  = 4 |A_x \cap A_y| - k.
\]
Thus if we let $i \in [m]$ and define $\F_i = \{ A_x : x \in S_i \}$ and $\G_i = \{ A_y : y \in T_i \}$, then we have $|A \cap B| \neq k/4$ for all $A \in \F_i, B \in \G_i$, so that we may apply \Cref{t:FR} with $\ell = k/4, \eta = 1/8$ to conclude that $|\F_i| |\G_i| \leq 2^{2k(1- \varepsilon)}$ for some fixed $\varepsilon > 0$. It follows that if $(S,T)$ is a dangerous pair, then
\[ 
|S||T| = \prod_{i=1}^{m}{|\F_i| |\G_i|} \leq 2^{2km(1- \varepsilon)}.
\]

We claim that for any $S, T \subseteq Q$ such that $u,v$ are non-orthogonal for all $u \in S, v \in T$, there exist $S' \supseteq S, T' \supseteq T$ such that $(S', T')$ is a dangerous pair. Indeed, given $S \subseteq Q$ and $i \in [m]$, let $\pi_i(S) = \{ u_i : u_1 \otimes \ldots \otimes u_m \in S\} $ denote the $i$th projection of $S$, and define $S' = \pi_1(S) \otimes \ldots \otimes \pi_m(S)$ and $T' = \pi_1(T) \otimes \ldots \otimes \pi_m(T)$. Clearly, $S' \supseteq S$ and $T' \supseteq T$, so suppose for sake of contradiction that $(S',T')$ is not a dangerous pair. Then there exists $i \in [m]$ and there exist $u_i \in  \pi_i(S), v_i \in \pi_i(T) $ such that $\inprod{u_i}{v_i} = 0$. By definition of $\pi_i$, there exist $u_j, v_j \in M_k$ for all $j \in [m] \backslash \{i\}$ such that $u = u_1 \otimes \ldots \otimes u_m \in S$ and $v = v_1 \otimes \ldots \otimes v_m \in T$, so that we conclude
\[ \inprod{u}{v} = \inprod{u_i}{v_i} \prod_{j \in [m] \backslash \{i\}}{\inprod{u_j}{v_j}} = 0,\] 
contradicting our choice of $S,T$.

Now let $k$ be the largest integer divisible by 4 such that $t \geq  \frac{2^{k+3}}{\varepsilon k} $ (so that $k = (1+o(1))\log{t}$) and let $m = \left \lceil \frac{2 \log{n} }{\varepsilon k} \right \rfloor$. It will suffice for us to show that with positive probability, a list of vectors $v_1, \ldots, v_n$ chosen independently and uniformly at random from $Q$ has no large overlap with any dangerous pairs, that is $|\{i : v_i \in S \}| < t$ or $|\{i : v_i \in T\}| < t$ for all dangerous pairs $(S,T)$. Indeed, given such a list, the previous claim implies that the graph $G$ with vertex set $V(G) = [n]$ and edge set $E(G) = \{ ij : \inprod{v_i}{v_j} \neq 0 \}$ is $K_{t,t}$-free and furthermore, $f(i) = v_i$ defines a faithful orthonormal representation of $G$ in $\R^{k^{m}}$, so that
\[
\msr_f(G) \leq k^{m} \leq 2^{ \frac{4 \log{n}}{\varepsilon k} \log{k}} = n^{\left( \frac{4}{\varepsilon}+o(1)\right) \frac{\log{\log{t}} }{ \log{t} }}.
\]
Moreover, using \Cref{lem_lovasz_edges}, \eqref{eq_theta_rank}, and the fact that $\msr(G) \leq \msr_f(G)$, it follows that the number of edges in $G$ is at least
\[
\left( \left( n / \msr_f(G) \right)^2 - n \right)/2 \geq n^{2 - \left( \frac{8}{\varepsilon}+o(1)\right) \frac{\log{\log{t}} }{ \log{t} } }.
\]

For any dangerous pair $(S,T)$, define $E_{S,T}$ to be the event that $|\{i : v_i \in S \}| \geq t$ and $|\{i : v_i \in T\}| \geq t$. Without loss of generality, we may assume that $|S| \leq |T|$ and so $|S|^2 \leq |S||T| \leq 2^{2km(1- \varepsilon)}$. Hence, a union bound implies that
\begin{align*}
\p[E_{S,T}] 
\leq \sum_{\substack{A \subseteq [n] \\ |A| = t } } \p[ v_i \in S \text{ for all } i \in A ] 
= \binom{n}{t} \left( \frac{|S|}{|Q|} \right)^t
\leq \left(\frac{n |S|}{|Q|}\right)^t 
&\leq \left( \frac{ 2^{ \varepsilon km / 2} 2^{km(1 - \varepsilon)} }{ 2^{km(1-o(1)) } } \right)^t\\
&= 2^{tkm( - \varepsilon/2 + o(1))}.
\end{align*}
We also observe that the number of dangerous pairs is at most $\left(2^{2^k}\right)^{2m} = 2^{2^{k+1}m}$, so that when $t$ is sufficiently large, applying another union bound yields the desired result
\[
\p\left[E_{S,T} \text{ for some dangerous pair } (S,T) \right] \leq 2^{2^{k+1}m} 2^{tkm( - \varepsilon / 2 + o(1))} < 1. \qedhere
\]
\end{proof}

We now turn our attention towards proving \Cref{thm_vec_lovasz} and \Cref{thm_ng}, for which we will need to establish an alternative characterization of the vector chromatic number. As previously mentioned, this parameter is known to be equivalent to Schrijver's theta function \cite{S79}, so our starting point is the following.

\begin{lemma} \label{lem_vec_equiv}
$\chivec(G)$ is the maximum of $\sum_{u,v \in V(G)}{M(u,v)}$ over all positive semidefinite matrices $M \in \R^{V(G) \times V(G)}$ satisfying $\tr(M) = 1$, $M(u,v) = 0$ if $uv \in E(G)$, and $M(u,v) \geq 0$ if $uv \notin E(G)$.
\end{lemma}

To state our alternative characterization, we need the following definition.

\begin{definition} \label{def_nonnegative}
Let $R$ be a Euclidean space with inner product $\inprod{\cdot}{\cdot}$. We call an orthonormal representation $f : V(G) \rightarrow R$ \emph{nonnegative} if $\inprod{f(u)}{f(v)} \geq 0$ for all vertices $u,v \in V(G)$.
\end{definition}

\begin{lemma} \label{lem_lower}
$\chivec(G)$ is maximum of $\sum_{v \in V(G)}{\inprod{f(v)}{x}^2}$ over all nonnegative orthonormal representations $f : V(G) \rightarrow R$ and all unit vectors $x \in R$.
\end{lemma}
\begin{proof}
Let $f : V(G) \rightarrow R$ be a nonnegative orthonormal representation of $G$. If we let $L$ be the matrix whose columns are the vectors $\{f(v): v \in V(G)\}$, then their Gram matrix $M = L^\intercal L$ has the same nonzero eigenvalues as $L L^\intercal = \sum_{v \in V(G)}{f(v) {f(v)}^\intercal}$ and so 
\begin{equation} \label{eq_duality}
\lambda_1(M) = \lambda_1\left(\sum_{v \in V(G)}{f(v) {f(v)}^\intercal} \right) 
= \max_{ ||x|| = 1} \sum_{v \in V(G)}{\inprod{f(v)}{x}^2}.
\end{equation}
Now let $y : V(G) \rightarrow \R$ be a unit eigenvector of $\lambda_1(M)$ and note that it has no negative coordinates by the Perron-Frobenius theorem. Also, let $M'$ be the Gram matrix of the vectors $\{ y(v) f(v) : v \in V(G) \}$, so that $M'_{u,v} = y(u) y(v) \inprod{f(u)}{f(v)}$. It is easy to see that $M'$ satisfies the conditions of \Cref{lem_vec_equiv} and thus, using \eqref{eq_duality} we conclude that
\[
\chivec(G) \geq \sum_{u,v \in V(G)}{M'_{u,v}}
= \sum_{u,v \in V(G)}{y(u) y(v) \inprod{f(u)}{f(v)}}
= \lambda_1(M)
= \max_{ ||x|| = 1} \sum_{v \in V(G)}{\inprod{f(v)}{x}^2}.
\]

The other direction follows by reversing the argument given above. Indeed, if we let $M'$ be a matrix satisfying the conditions of \Cref{lem_vec_equiv} such that $\chivec(G) = \sum_{u,v \in V(G)}{M'(u,v)}$, then since $M'$ is positive semidefinite, it is the Gram matrix of some set of vectors $\{g(v) : v \in V(G)\}$, i.e.\ $M'(u,v) = \inprod{g(u)}{g(v)}$ for $u,v \in V(G)$. Then we define $f(v) = g(v)/||g(v)||$ (if $g(v) = 0$, we can set $f(v)$ to be any unit vector) and observe that $f$ is a nonnegative orthonormal representation of $G$. Letting $M$ be the Gram matrix of $\{f(v) : v \in V(G)\}$ and defining $y(v) = ||g(v)||$ for all $v \in V(G)$, we have that $y$ is a unit vector and thus,
\[
\lambda_1(M) \geq y^\intercal M y = \sum_{u,v \in V(G)}{y(u)y(v) \inprod{f(u)}{f(v)}} = \sum_{u,v \in V(G)}{M'(u,v)} = \chivec(G).
\]
The result now follows from \eqref{eq_duality}.
\end{proof}

In the following, given matrices $X,Y \in \R^{I \times I}$, we let $\inprod{X}{Y}_F = \sum_{i,j \in I}{X(i,j)Y(i,j)}$ denote their Frobenius inner product. We also let $\R^{r \times r}$ denote the space of $r \times r$ real matrices.

\begin{proof}[Proof of \Cref{thm_vec_lovasz}]
Let $f : V(G) \rightarrow \R^r$ be an orthonormal representation of $G$ whose Gram matrix $M$ satisfies $\lambda_1 = \lambda_1(M) = \vartheta\left(\overline{G}\right)$ and let $x : V(G) \rightarrow \R$ be a corresponding unit eigenvector. We assign to each vertex $v \in V(G)$, the matrix $g(v) = f(v) {f(v)}^\intercal$ and observe that $\inprod{g(u)}{g(v)}_F = \inprod{f(u)}{f(v)}^2$ for all $u, v \in V(G)$, so that $g : V(G) \rightarrow \R^{r \times r}$ is a nonnegative orthonormal representation of $G$ when $\R^{r \times r}$ is equipped with the Frobenius inner product. Now let $L$ be the matrix whose columns are the vectors $\{f(v): v \in V(G)\}$ and define $B = \frac{1}{\lambda_1} L x (Lx)^\intercal$. Then we have 
$\inprod{B}{B}_F = \frac{\inprod{Lx}{Lx}^2}{\lambda_1^2} = \frac{(x^\intercal M x)^2}{\lambda_1^2} = 1$
and
\[ 
\inprod{g(v)}{B}_F = \frac{\inprod{f(v)}{Lx}^2}{\lambda_1 }
= \frac{\inprod{L e_v}{Lx}^2}{\lambda_1 } 
= \frac{\left( e_v^\intercal M x \right)^2}{\lambda_1 }
= \lambda_1 x(v)^2
\]
for all $v \in V(G)$. Thus we may apply \Cref{lem_lower} and Cauchy-Schwartz to obtain the desired bound
\[
\chivec(G) \geq \sum_{v \in V(G)}{\inprod{g(v)}{B}_F^2}
= \lambda_1^2 \sum_{v \in V(G)}{x(v)^4}
\geq \frac{\lambda_1^2}{n} \left( \sum_{v \in V(G)}{x(v)^2} \right)^2
= \frac{\lambda_1^2}{n}. \qedhere
\]
\end{proof}

\begin{proof}[Proof of \Cref{thm_ng}]
Let $r = \msr(G)$ and let $f : V(G) \rightarrow \R^r$ be an orthonormal representation of $G$. As in the proof of \Cref{thm_vec_lovasz}, we assign the matrix $g(v) = f(v) f(v)^\intercal$ to each vertex $v \in V(G)$ and observe that $g : V(G) \rightarrow \R^{r \times r}$ is a nonnegative orthonormal representation of $G$. Thus, if we let $B = \frac{1}{\sqrt{r}} I$ where $I \in \R^{V(G) \times V(G)}$ is the identity matrix, then $\inprod{B}{B}_F = \tr(I)/r = 1$ and so, \Cref{lem_lower} yields the desired bound 
\[
\chivec(G) \geq \sum_{v \in V(G)}{\inprod{g(v)}{B}_F^2}
= \frac{1}{r} \sum_{v \in V(G)}{\tr\left(f(v) {f(v)}^\intercal \right)^2}
= \frac{n}{r}. \qedhere
\]
\end{proof}

It now remains to establish \Cref{thm_rank_plus} and \Cref{cor_extension}. We first demonstrate that \Cref{cor_extension} follows from \Cref{thm_rank_plus} via a result of Yannakakis \cite{Y91}, who showed that the extension complexity of a polytope equals the nonnegative rank of any slack matrix of this polytope. Indeed, Kwan, Sauermann, and Zhao observed that his result directly implies the following lemma.

\begin{lemma}[Lemma 2.4 in \cite{KSZ22}] \label{lem_yan}
For $n, r \in \N$, let $f_{\xc}(n, r)$ be the maximum extension complexity of a polytope with dimension at most $d$ and at most $n$ facets and let $f_{+}(n,r)$ be the maximum nonnegative rank of a nonnegative matrix with at most $n$ columns and rank at most $r$. Then for all $r \geq 2$, we have 
\[
f_{\xc}(n, r-1) = f_{+}(n,r).
\]
\end{lemma}

\begin{proof}[Proof of \Cref{cor_extension}]
Let $M$ be the $n \times n$ nonnegative matrix given by \Cref{thm_rank_plus}, so that $r = \rk(M) \leq n^{C\sqrt{ \frac{\log{\log{n}}}{\log{n}} } }$ and $\rkplus(M) \geq n^{1 - C\sqrt{ \frac{\log{\log{n}}}{\log{n}} } }$. Using \Cref{lem_yan}, it follows that 
\[
f_{\xc}(n, r-1) = f_{+}(n,r) \geq n^{1 - C\sqrt{ \frac{\log{\log{n}}}{\log{n}} } },
\]
so there exists a polytope $P$ with at most $n$ facets and dimension at most $n^{C\sqrt{ \frac{\log{\log{n}}}{\log{n}} } } $ such that $\xc(P) \geq  n^{1 - C\sqrt{ \frac{\log{\log{n}}}{\log{n}} } }$. The result now follows by considering the polar dual of $P$.
\end{proof}

In order to prove \Cref{thm_rank_plus}, we need a way to lower bound the nonnegative rank of a matrix. It turns out that a combinatorial parameter known as the rectangle covering number provides such a bound. Intuitively, it is the minimum number of rectangles necessary to cover the support of the given matrix. 

\begin{definition}
Let $I$ and $J$ be finite sets. For any $I' \subseteq I$ and $J' \subseteq J$, we call $I' \times J'$ a \emph{rectangle}. 
For any matrix $M$ indexed by $I \times J$, we define the \emph{rectangle covering number} $\rc(M)$ to be the minimum $k$ such that there exist rectangles $R_1, \ldots, R_k$ satisfying $\bigcup_{i=1}^{r}{R_i} = \{ (i,j) \in I \times J : M(i,j) \neq 0 \}$). 
\end{definition}

\begin{lemma}[equation (2) from \cite{FKPT13}]\label{lem_rectangle_bound}
For any nonnegative matrix $M$,
\[
\rkplus(M) \geq \rc(M).
\]
\end{lemma}

\begin{proof}[Proof of \Cref{thm_rank_plus}]
First note that we may assume $n$ is sufficiently large, since otherwise the result holds trivially by taking $C$ sufficiently large. Let $t$ be the largest integer smaller than $n^{\sqrt{ \frac{\log{\log{n}}}{\log{n}} } }$. \Cref{thm_main} implies that there exists a graph $G$ with $n$ vertices and at least $n^{2 - C' \frac{\log{\log{t}}}{ \log{t} } }$ edges having a faithful orthonormal representation $f : V(G) \rightarrow \R^r$ such that $r \leq n^{C' \frac{\log{\log{t}}}{ \log{t} } }$, where $C' > 0$ is a constant. Now define $g : V(G) \rightarrow \R^{r \times r}$ by $g(v) = f(v) {f(v)}^\intercal$ and consider the corresponding Gram matrix $M \in \R^{V(G) \times V(G)}$ defined by $M(u,v) = \inprod{g(u)}{g(v)}_F = \inprod{f(u)}{f(v)}^2$ for $u,v \in V(G)$. Clearly $M$ is nonnegative and since the dimension of $\R^{r \times r}$ is $r^2$, we have $\rk(M) \leq r^2$. 

It remains to lower bound $\rkplus(M)$ and in view of \Cref{lem_rectangle_bound}, it will suffice to bound $\rc(M)$. To this end, observe that since $G$ is $K_{t,t}$-free, we have that any rectangle $I \times J \subseteq V(G) \times V(G)$ with $M(u,v) \neq 0$ for all $(u,v) \in I \times J$ satisfies $\min(| I |, | J |) < 2t$. Indeed, otherwise the subgraph spanned by $I \cup J$ would contain a copy of $K_{t,t}$. It follows that any rectangle contained in $\{(u,v) \in V(G) \times V(G) : M(u,v) \neq 0 \}$ has cardinality at most $2t n$. Moreover, since $f$ is faithful, $M$ has at least $2n^{2 - C' \frac{\log{\log{t}}}{ \log{t} } }$ nonzero entries and so, we conclude that
\[
\rc(M) \geq \frac{2 n^{2 - C' \frac{\log{\log{t}}}{ \log{t} } } }{ 2 t n} \geq n^{1 - (C' + 1)\sqrt{ \frac{\log{\log{n}}}{\log{n}} } },
\]
as desired.
\end{proof}

\end{document}